\documentclass[letterpaper, 10 pt, conference]{ieeeconf}  

\IEEEoverridecommandlockouts                              

\usepackage{graphicx}
\graphicspath{ {Pics/} }

\usepackage{amsmath}
\usepackage{float}
\usepackage{amssymb}
\usepackage{amsthm}
\usepackage{url}
\usepackage{xcolor}
\usepackage[
backend=biber,
style=numeric,
sorting=none
]{biblatex}

\usepackage[colorlinks, allcolors=blue]{hyperref}

\usepackage{algorithm}
\usepackage{algpseudocode}
\newtheorem{lemma}{Lemma}
\newtheorem{definition}{Definition}
\newtheorem{theorem}{Theorem}
\newtheorem{rem}{Remark}

\newcommand{\df}{\overset{\text{def}}{=}}
\newcommand{\defe}{\overset{\mathrm{def}}{=}}

\allowdisplaybreaks

\title{\LARGE \bf
	Minimisation of Polyak-\L{}ojasewicz Functions Using Random Zeroth-Order Oracles
}

\author{Amir Ali Farzin and Iman Shames
	\thanks{The authors are with CIICADA Lab, School of Engineering, ANU
		{\tt\small \{amirali.farzin,iman.shames\}@anu.edu.au}}%
}%

\begin{document}
	\maketitle
	
	\begin{abstract} The application of a zeroth-order scheme for minimising Polyak-\L{}ojasewicz (PL) functions is considered. The framework is based on exploiting a random oracle to estimate the function gradient. The convergence of the algorithm to a global minimum in the unconstrained case and to a neighbourhood of the global minimum in the constrained case along with their corresponding complexity bounds are presented. The theoretical results are demonstrated via numerical examples.
	\end{abstract}

	\section{Introduction}
	Zeroth-order (ZO) or derivative-free optimisation schemes are of interest when the gradient (or subgradient in case of non-differentiable cost functions) information is not readily available. A common scenario is the case where the value of the cost function, and not its higher order derivatives, is the only information available to the solver \cite{maass2021zeroth} and \cite{marco2016automatic}. Sometimes, even if the gradient value is theoretically available, evaluating the gradient might incur high computational costs \cite{bottou2018optimization}. ZO methods provide a way forward for solving such optimisation problems as well. The majority of existing ZO methods aim to construct an estimate of the gradient of the function and use this estimate as a surrogate for the gradient. The method analysed in this paper is no exception to this general trend.
	
	Various ZO optimisation methods have been designed and analysed for different problem classes. In \cite{ghadimi2016mini}, a constrained stochastic composite optimisation problem was studied where the function is possibly non-convex. The proposed algorithm in \cite{ghadimi2016mini} relied on the existence of an unbiased variance-bounded estimator of the gradient. In \cite{nesterov_random_2017}, the authors considered the unconstrained zeroth-order optimisation problem and defined a random oracle which was an unbiased variance-bounded estimator of the gradient of a smoothed version of the original function.  
	
	The Polyak-\L{}ojasiewicz (PL) inequality was first introduced in \cite{polyak1964gradient} and the convergence of gradient descent method under PL assumption was first proved there. 
	It is observed that a range of different cost functions satisfy the PL condition \cite{liu2022loss}. 
	Karimi et al. used the PL inequality to provide a new proof technique for analysing various first-order gradient descent methods and used proximal PL condition to analyse non-smooth cases \cite{karimi2016linear}.  In \cite{anagnostidis2021direct}, a variant of the direct search method was employed to solve stochastic minimisation and saddle point problems. The direct search algorithm which is a derivative-free scheme and obtained the complexity bounds for the convergence of PL functions.
	
	In this paper, we aim to fill a gap in the existing literature on random zeroth order oracles. We specifically leverage the class of random method proposed in \cite{nesterov_random_2017} for optimisation problems, and establish its performance for the case where the cost functions satisfy the (proximal) PL condition. Specifically, we establish the convergence properties and the complexity bounds of these methods for both constrained and unconstrained cases. By doing so, we hope to shed some light on the behaviour of such algorithms when applied to a class of benignly nonconvex problems.
	
	The outline of this paper is as follows. In Section \ref{sec:prelim} we introduce the necessary background information. The problems of interest are outlined in \ref{sec:prob}. The convergence properties and the complexity of the algorithms for solving the problems of interest are presented in Section~\ref{sec:main}. Numerical examples are presented in \ref{sec:nuemrical} and conclusions and possible future research directions come in the end.

	\section{Preliminaries}\label{sec:prelim}
	In this section we provide the necessary definitions and background material required for presenting the results of this paper.
	
	Consider a function $f: \mathbb{R}^n\rightarrow \mathbb{R}$. The gaussian smoothed version of $f$, termed $f_{\mu}(x)$, is defined below:
	\begin{align}\label{eq:eq21}
		\begin{split}
			&f_\mu(x) \defe \frac{1}{\kappa}\underset{E}{\int}f(x+\mu u)e^{-\frac{1}{2}||u||^2}\mathrm{d}u,\\&\kappa \defe\underset{E}{\int}e^{-\frac{1}{2}||u||^2}\mathrm{d}u=\frac{(2\pi)^{n/2}}{[det B]^{\frac{1}{2}}},
		\end{split}
	\end{align}
	where vector $u \in R^n$ is sampled from zero mean Gaussian distribution with positive definite correlation operator $B^{-1}$ and the positive scalar smoothing parameter $\mu$. 
	Define the random oracle $g_\mu$ as
	\begin{equation}\label{eq:eq23}
		g_\mu(x) \defe \frac{f(x+\mu u)-f(x)}{\mu} Bu,
	\end{equation}
	where $u$ and $B$ are defined above. 
	The projection operator on a convex set $\mathcal{Z}$ is defined as $$\mathrm{Proj}_{\mathcal{Z}}(x) \df \arg\underset{z\in\mathcal{Z}}{\min} \|z-x\|^2,$$ 
	where $\|\cdot\|$ is the Euclidean norm of its argument if it is a vector, and the corresponding induced norm if the argument is a matrix.
	\begin{definition}[$C^{1,1}$ Functions]\label{def:lip_grad}
		The differentiable function $f(x):D\rightarrow \mathbb{R}$ with $D\subseteq \mathbb{R}^n$ as its domain is in $C^{1,1}$ if its gradient is Lipschitz, i.e.,
		\begin{equation}
			||\nabla f(x)-\nabla f(y)||\leq L_1(f)||x-y||,\forall x,y\in D,
		\end{equation}
		where $L_1(f)>0$ is the \emph{gradient Lipschitz constant}.
	\end{definition}
	\begin{rem}
		Any  $f(x) \in C^{1,1}$ satisfies
		\begin{equation} \label{eq:eq2}
			f(y)\leq f(x) + \langle \nabla f(x), y-x\rangle+\frac{L_1(f)}{2}||y-x||^2.
		\end{equation} 
	\end{rem}
	\begin{definition}[PL Functions \cite{polyak1964gradient}]
		The differentiable function $f(x):D\rightarrow \mathbb{R}$ with $D\subseteq \mathbb{R}^n$ as its domain is termed a \emph{PL function} if it satisfies the Polyak-\L{}ojasewicz (PL) condition, i.e.,
		\begin{equation}\label{eq:eq3}
			\dfrac{1}{2}||\nabla f(x)||^2 \geq l (f(x) - f^\ast),\; \forall x\in D,
		\end{equation}  
		where $l>0$ is the \emph{PL constant} and $f^\ast=f(x^\ast).$ 
	\end{definition} 
	We have the following result for PL functions.
	\begin{lemma}\label{lm:lm1}
		PL inequality implies that every stationary point of the function is a global minimum.
	\end{lemma}
	\begin{proof} 
		Assume that $x$ is an arbitrary stationary point, i.e., $||\nabla f(x)||=0$. Substituting $x$ in \eqref{eq:eq3} yields $$f(x)-f^\ast\leq0\implies f(x) =f^\ast.$$ Hence, any stationary point corresponds to the global minimum.
	\end{proof}
	\begin{definition}[Proximal PL Functions]\label{as}
		Consider the function $F(x)=f(x)+h(x)$ where $f(x):D\rightarrow \mathbb{R}$ with $D\subseteq \mathbb{R}^n$ is a differentiable function over as its domain and $h(x)$ is possibly nondiffrentiable. The function $F(x)$ is a proximal PL (PPL) function if it satisfies the PPL condition on a set $\mathcal{X}$ with positive constant $l$, for all $x \in \mathcal{X}$, i.e.,
		\begin{align}\label{eq:PPL}
			\frac{1}{2}Q(x,L_{1}(f)) \geq l (F(x)-F(x^\ast)),
		\end{align}
		where
		\begin{align}\label{eq:PPL_quad}
			\begin{split}
				Q(x,a) \defe &-2a \underset{z\in\mathcal{X}}{\min} \{\frac{a}{2}||z-x||^2+\\&\langle\nabla f(x),z-x\rangle+h(z)-h(x)\}
			\end{split}
		\end{align}
		with $a$ being a positive scalar.
	\end{definition}
	The Proximal PL condition is a generalisation of the PL condition \eqref{eq:eq3}. To see more examples, refer to \cite{karimi2016linear}.
	
	To state complexity results we use the big-O notation in the sense defined below.
	\begin{definition}[The big O-notation]
		Suppose $f(x)$ and $g(x)$ are two positive scalar functions defined on some subset
		of the real numbers. We write 	$f(x) = \mathcal{O}(g(x))$, and say $f(x)$ is in the order of $g(x)$, if and only if there exist constants $\bar{K}$ and $M$ such that 	$f(x) \leq M g(x)$ for all $x>K$.
	\end{definition}
	\section{Problems of Interest}\label{sec:prob}
	In this paper we study the performance of a random zeroth-order method for optimising PL functions for both unconstrained and constrained cases. Specifically, first, we study the following unconstrained problem
	\begin{equation}\label{eq:eq1}
		\min_{x\in R^n} \quad f(x)
	\end{equation}
	where $f:\mathbb{R}^n\rightarrow \mathbb{R}$ is in $C^{1,1}$, satisfies the PL condition \eqref{eq:eq3}, and is bounded below.

	
	Later, we will focus on the following constrained problem:
	\begin{equation}\label{eq:eq87} 
		\min_{x \in \mathcal{X}} \quad f(x),
	\end{equation}
	where $\mathcal{X}$ is a compact convex set in $\mathbb{R}^n$ with diameter $d_x$ and $f(x)$ is in $C^{1,1}$ and $F(x)=f(x)+\mathcal{I}_{\mathcal{X}}(x)$ satisfies \eqref{eq:PPL} where $\mathcal{I}_{\mathcal{X}}(x)$ is the indicator function of set $\mathcal{X}$, i.e.,
	\[\mathcal{I}_{\mathcal{X}}(x)=\begin{cases}
		0 & x \in \mathcal{X}\\
		\infty & x \not\in \mathcal{X}.
	\end{cases}\]
	The boundedness of $\mathcal{X}$ and continuity of $f(x)$ guarantees that the problem has a solution.

	\section{Main Results}\label{sec:main}
	In this section, we establish the convergence properties and the complexity bounds of a well-known class of random zeroth-order algorithms proposed by Nesterov and Spokoiny in \cite{nesterov_random_2017} for minimising (proximal) PL functions. We consider both unconstrained and constrained cases.  
	
	\subsection{Unconstrained Problem}
	In this subsection, we study the unconstrained problem \eqref{eq:eq1}. The framework introduced in \cite{nesterov_random_2017} is recalled in Algorithm~\ref{alg:RM}, where $x_0$ is the initial guess, $h_k$ is the step size and $N$ is the number of iterations. 
	\begin{algorithm}[ht]
		\caption{$\mathrm{RS}_\mu$}\label{alg:RM}
		\begin{algorithmic}[1]
			\State Input: $x_0,\mu,h_k,N$
			\For {$k = 0,\dots, N$}
			\State Generate $u$
			\State Calculate $g_\mu(x_k)$ using \eqref{eq:eq23}
			\State $x_{k+1} = x_k - h_kg_\mu(x_k)$ 
			\EndFor
			\State return $x_{N}$
		\end{algorithmic}
	\end{algorithm}
	
	The following theorem characterises the convergence properties of Algorithm~\ref{alg:RM} applied to problem \eqref{eq:eq1}.
	\begin{theorem}\label{thm:unconstrained}
		Let the sequence $\{x_k\}_{k\geq0}$ be generated by Algorithm~\ref{alg:RM} (RS$_\mu$), where $f(x)\in C^{1,1}$ and satisfies PL condition. Then,  for any $N\geq0$, with $h_k = \frac{1}{4(n+4)L_{1}(f)}$, we have
		\begin{align}
			\begin{split}
				\frac{1}{N+1}&\sum_{k=0}^{N}(\Phi_k-f^\ast)\leq \frac{8(n+4)L_1(f)}{l} \left [\frac{f(x_0)-f^\ast}{N+1} \right.\\&+ \left . \frac{3\mu^2(n+4)}{32}L_1(f) \right ] + \frac{\mu^2}{4l}L_1^2(f)(n+6)^3,
			\end{split}	
		\end{align}
		where $\Phi_k \overset{def}{=} E_{\mathcal{U}_{k}}[f(x_k)]$ for all $k\geq 1$ and $\Phi_0 = f(x_0)$. Also, $\mathcal{U}_{k}=\{u_{0},u_{1},\dots,u_{k-1}\}$ and $n$ is the dimension of $x$.
	\end{theorem}
	\begin{proof} From \cite[Equation (12)]{nesterov_random_2017}, we know that if $f \in C^{1,1}$, then $f_\mu \in C^{1,1}$ and $L_1(f_\mu)\leq L_1(f)$. Thus writing \eqref{eq:eq2} for $f_\mu(x)$ at points $x_k$ and $x_{k+1}$
		\begin{align}\label{eq:eq8}
			\begin{split}
				f_\mu(x_{k+1}) \leq& f_\mu(x_k) - h_k\langle\nabla f_\mu(x_k),g_\mu(x_k)\rangle\\&+\frac{1}{2}h_k^2L_1(f_\mu)||g_\mu(x_k)||^2.
			\end{split}
		\end{align}
		From {\cite[Equation (21)]{nesterov_random_2017}}, we know for a function $f(x)$ we have $$	\nabla f_{\mu}(x) = 	\frac{1}{\kappa}\int \frac{f(x+\mu u)-f(x)}{\mu}e^{-\frac{1}{2}||u||^2}Bu\mathrm{d}u.$$
		From the definition of $g_\mu(x)$ and the term obtained for $\nabla f_{\mu}(x)$, we have $E_{u}[g_{\mu}(x)] = \nabla f_{\mu}(x)$ (it means $g_{\mu}(x)$ is an unbiased estimator of $\nabla f_{\mu}(x)$).
		Taking the expectation in $u_k$ yields
		\begin{align} \label{eq:eq9}
			\begin{split}
				E_{u_k}(f_\mu(x_{k+1})) \leq& f_\mu(x_k) - h_k||\nabla f_\mu(x_k)||^2\\&+\frac{1}{2}h_k^2L_1(f_\mu)E_{u_k}(||g_\mu(x_k)||^2).
			\end{split}
		\end{align}
		For a function $f\in C^{1,1}$ from {\cite[Lemma 5]{nesterov_random_2017}}, we have
		\begin{align}\label{eq:eq6}
			\begin{split}
				E_{u}[||g_\mu(x)||^2] \leq& 4(n+4)||\nabla f_\mu(x)||^2\\&+3\mu^2L_1^2(f)(n+4)^3
			\end{split}
		\end{align}
		Substituting \eqref{eq:eq6} in \eqref{eq:eq9} and noting $L_1(f_\mu)\leq L_1(f)$, we obtain
		\begin{align}
			\begin{split}
				E_{u_k}(f_\mu(x_{k+1}))& \leq f_\mu(x_k) - h_k||\nabla f_\mu(x_k)||^2+\\&\frac{1}{2}h_k^2L_1(f)\left(4(n+4)||\nabla f_\mu(x_k)||^2\right .\\&\left.+3\mu^2L_1^2(f)(n+4)^3\right).
			\end{split}
		\end{align}
		Fixing $h_k = \hat{h} \overset{def}{=} \frac{1}{4(n+4)L_1(f)}$:
		\begin{align}
			\begin{split}
				E_{u_k}(f_\mu(x_{k+1}))&\leq f_\mu(x_k) - \frac{1}{2}\hat{h}||\nabla f_\mu(x_k)||^2 \\&+ \frac{3\mu^2}{32}L_1(f)(n+4).
			\end{split}
		\end{align}
		Taking the expectation of this inequality with respect to $\mathcal{U}_{k-1} = \{u_0,u_1,\dots,u_{k-1}\}$, we obtain
		\begin{equation}\label{eq:eq12}
			\Phi_{k+1} \leq \Phi_k -\frac{1}{2}\hat{h}\Xi_k^2 + \frac{3\mu^2(n+4)}{32}L_1(f),
		\end{equation}
		where $\Xi_k^2 = E_{\mathcal{U}_k}(||\nabla f_\mu(x_k)||^2)$. Considering $f^\ast\leq f(x_{N+1})$ and summing \eqref{eq:eq12} from $k = 0$ to $k=N$ and divide it by $N+1$, we get
		\begin{align}\label{eq:eq13}
			\begin{split}
				\frac{1}{N+1} \sum_{k=0}^{N}\Xi_k^2 &\leq 8(n+4)L_1(f)[\frac{f(x_0)-f^\ast}{N+1}\\&+\frac{3\mu^2(n+4)}{32}L_1(f)],
			\end{split}
		\end{align}
		From \cite[Lemma 4]{nesterov_random_2017}, it is known that for a function $f\in C^{1,1}$
		\begin{equation}\label{eq:eq7}
			||\nabla f(x)||^2 \leq 2||\nabla f_\mu(x)||^2 + 	\frac{\mu^2}{2}L_1^2(f)(n+6)^3.
		\end{equation}
		From \eqref{eq:eq3} and \eqref{eq:eq7} one concludes
		\begin{align}\label{eq:eq14}
			\begin{split}
				2l(f(x_k)-&f^\ast)\leq ||\nabla f(x_k)||^2 \\&\leq 2||\nabla f_\mu(x_k)||^2 + \frac{\mu^2}{2}L_1^2(f)(n+6)^3.
			\end{split}
		\end{align}
		Taking the expectation of this inequality with respect to $\mathcal{U}_{k}$
		\begin{equation}\label{eq:eq15}
			E_{\mathcal{U}_k}(2l(f(x_k)-f^\ast))\leq \theta_k^2 \leq 2\Xi_k^2 + \frac{\mu^2}{2}L_1^2(f)(n+6)^3,
		\end{equation}
		where $\theta_k^2 = E_{\mathcal{U}_k}(||\nabla f(x_k)||^2)$. Summing the inequality from $k = 0$ to $k=N$ and dividing it by $N+1$, yields
		\begin{align}\label{eq:eq16}
			\begin{split}
				\frac{1}{N+1}\sum_{k=0}^{N}&(\Phi_k-f^\ast)\leq \frac{1}{2l(N+1)}\sum_{k=0}^{N}\theta_k^2 \\&\leq \frac{1}{l(N+1)}\sum_{k=0}^{N}\Xi_k^2 + \frac{\mu^2}{4l}L_1^2(f)(n+6)^3.
			\end{split}
		\end{align}
	\end{proof}
	In a practical implementation, we might be interested in identifying the ``best'' solution guess at any step $N$. To this aim, define $\hat{x}_N \overset{\mathrm{def}}{=} \arg\underset{x}{\min}[f(x):x\in\{x_0,\dots,x_N\}]$. Hence,
	\begin{equation}\label{eq:eq17}
		E_{\mathcal{U}_{N-1}}(f(\hat{x}_N)-f^\ast )\leq 	\frac{1}{N+1}\sum_{k=0}^{N}(\Phi_k-f^\ast).
	\end{equation}
	\begin{rem}[$\mathrm{RS}_{\mu}$ Complexity, Parameter Selection, and Solution Error Bound]\label{remCU}
		If $\mu$ is in the order of $\mathcal{O}(\frac{\sqrt{l\epsilon}}{n^{3/2} L_1(f)})$ 
		and $N$ is in the order of $\mathcal{O}(\frac{nL_1(f)}{l\epsilon})$, it is guaranteed that
		$E_{\mathcal{U}_{N-1}}(f(\hat{x}_N))-f^\ast \leq \epsilon$ for some positive scalar $\epsilon$. 
	\end{rem}

	In comparison with the non-convex smooth case in \cite{nesterov_random_2017}, from the properties of PL functions the convergence is to a global minimum (see Lemma \ref{lm:lm1}). Also, comparing the bounds in these two cases, for the case where $l>1$ one can obtain the same error bound, i.e. $\epsilon$, with fewer iterations. The number of iterations is inversely proportional with $l$.  
	\subsection{Constrained Problem}
	Now we will focus on the problem \eqref{eq:eq87}. This problem can be reformulated as  
	$$\underset{x \in R^n}{\min} \;f(x)+\mathcal{I}_{\mathcal{X}}(x),$$ where $\mathcal{I}_{\mathcal{X}}(x)$ is the indicator function of $\mathcal{X}$ . The new scheme for constrained problem is defined in Algorithm~\ref{alg:RMc}.
	\begin{algorithm}[ht]
		\caption{$\mathrm{RSc}_\mu$}\label{alg:RMc}
		\begin{algorithmic}[1]
			\State Input: $x_0,h_k,\mu,N$
			\For {$k = 0,\dots, N$}
			\State Generate $u$
			\State Calculate $g_\mu(x_k)$ using \eqref{eq:eq23}
			\State $\bar{x}_{k+1} = x_k - h_kg_\mu(x_k)$
			\State $x_{k+1} = \mathrm{Proj}_{\mathcal{X}}(\bar{x}_{k+1})$
			\EndFor
			\State return $x_{N}$
		\end{algorithmic}
	\end{algorithm}
	In Algorithm \ref{alg:RMc}, the projection is used to ensure that the output sequence is completely in the feasible set.
	
	Before proceeding further, we define the following auxiliary variables:
	\begin{equation}\label{eq:eq91}
		P_\mathcal{X}(x,g(x),h) \df \frac{1}{h}[x-\mathrm{Proj}_{\mathcal{X}}(x-hg(x))] ,
	\end{equation}
	\begin{equation}\label{eq:eq92}
		s_k \df P_\mathcal{X}(x_k,g_{\mu}(x_k),h_k),
	\end{equation}
	\begin{equation}\label{eq:eq93}
		v_k \df P_\mathcal{X}(x_k,\nabla f_{\mu}(x_k),h_k).
	\end{equation}
	In the unconstrained case we had $x_{k+1}=x_k-h_kg_{\mu}(x_k)$. In the constrained case, from Algorithm \ref{alg:RMc}, \eqref{eq:eq91}, and \eqref{eq:eq92}, we can see that $x_{k+1}=x_k-h_ks_k.$
	
	
	%
	

	Before stating the main result, in what follows, we propose a lower bound for the value of an operator that plays a crucial role in proving the main result of this section.
	\begin{lemma}\label{lmPPL}
		Consider problem \eqref{eq:eq87} where   $f(x)\in C^{1,1}$ is a proximal PL function in the sense of Definition \ref{as}, we have 
		\begin{align}\label{eq:eqPPL}
			\begin{split}
				T(x_k,L_{1}&(f)) \geq 2l (f(x_k)-f(x^\ast))\\&-\mu L_{1}(f)^2(n+3)^{3/2}d_x- 2L_{1}(f)d_x||\xi_{k}||,
			\end{split}
		\end{align}
		where 
		\begin{align}
			T(x,a) \df -2a \underset{z\in\mathcal{X}}{\min}& \label{eq:T_def} \{\frac{a}{2}||z-x||^2\\&+\langle g_{\mu}(x),z-x\rangle+\mathcal{I}_{\mathcal{X}}(z)-\mathcal{I}_{\mathcal{X}}(x)\}, \notag
		\end{align}for a positive scalar $a$. 
	\end{lemma}
	\begin{proof}
		From the definition of PPL functions, we have
		\begin{align*}
			\begin{split}
				&2l(f(x_k)-f(x^\ast))\leq Q(x_k,L_{1}(f))\\&
				\leq -2L_{1}(f) \underset{z\in\mathcal{X}}{\min} (\frac{L_{1}(f)}{2}||z-x_k||^2\\&+\langle\nabla f(x_k),z-x_k\rangle)\\&\leq
				-2L_{1}(f) \underset{z\in\mathcal{X}}{\min} (\frac{L_{1}(f)}{2}||z-x_k||^2+\langle\nabla f_{\mu}(x_k)\\&,z-x_k\rangle+\langle\nabla f(x_k)-\nabla f_{\mu}(x_k),z-x_k\rangle)\\&\leq
				-2L_{1}(f) \underset{z\in\mathcal{X}}{\min} (\frac{L_{1}(f)}{2}||z-x_k||^2+\langle\nabla f_{\mu}(x_k)\\&,z-x_k\rangle-||\nabla f(x_k)-\nabla f_{\mu}(x_k)||||z-x_k||),
			\end{split}
		\end{align*}
		where the second inequality is a consequence of evaluating \eqref{eq:PPL_quad} for $x_k\in\mathcal{X}$ and noting that $\mathcal{I}_{\mathcal{X}}(x_k) =0$.
		From {\cite[Lemma 3]{nesterov_random_2017}}, for a function $f\in C^{1,1}$ we have
		\begin{equation}\label{eq:lm}
			||\nabla f(x) - \nabla f_\mu(x)|| \leq \frac{\mu}{2}L_1(f)(n+3)^{3/2}.
		\end{equation}
		From \eqref{eq:lm} and $||z-x_k||\leq d_x$ for all $z\in\mathcal{X}$, we have 
		\begin{align*}
			\begin{split}
				&2l(f(x_k)-f(x^\ast))\leq Q(x_k,L_{1}(f))\\&\leq
				-2L_{1}(f) \underset{z\in\mathcal{X}}{\min} (\frac{L_{1}(f)}{2}||z-x_k||^2\\&+\langle\nabla f_{\mu}(x_k),z-x_k\rangle-\frac{\mu L_{1}(f)(n+3)^{3/2}d_x}{2})\\&\leq -2L_{1}(f) \underset{z\in\mathcal{X}}{\min} (\frac{L_{1}(f)}{2}||z-x_k||^2+\langle g_{\mu}(x_k)\\&,z-x_k\rangle-||\xi_{k}||d_x-\frac{\mu L_{1}(f)(n+3)^{3/2}d_x}{2}).
			\end{split}
		\end{align*}
		where $\xi_{k} \overset{\mathrm{def}}{=} g_{\mu}(x_k) -\nabla f_{\mu}(x_k)$.	Rearranging above terms completes the proof.
	\end{proof}
	Before stating the main result regarding the performance of Algorithm~\ref{alg:RMc}, we present the following lemma on the variance of the random oracle $g_{\mu}(x)$.
	\begin{lemma}\label{lm20}
		Random oracle $g_{\mu}(x)$ is a variance bounded unbiased estimator of $\nabla f_{\mu}(x)$. That is $$E_{u}[||g_{\mu}(x_k)- \nabla f_{\mu}(x_k)||^2]\leq \sigma_k^2,\;\; \sigma_k\geq0.$$
	\end{lemma}
	\begin{proof}
		We know that $E_{u}[g_{\mu}(x_k)]= \nabla f_{\mu}(x_k)$, so we have $E_{u}[||g_{\mu}(x_k)- \nabla f_{\mu}(x_k)||^2]\leq E_{u}[||g_{\mu}(x_k)||^2]\leq \sigma_k^2.$
	\end{proof}
	\begin{rem}\label{remg}
		An upper bound for $E_{u}[||g_{\mu}(x_k)||^2]$ can be obtained. For example, from {\cite[Theorem 4]{nesterov_random_2017}} we know for a function $f\in C^{0,0}$ we have $E_{u}[||g_{\mu}(x)||^2]\leq L_{0}(f)^2(n+4)^2$ and for a function $f\in C^{1,1}$ we have $E_{u}[||g_{\mu}(x)||^2]\leq \frac{\mu^2}{2}L_1^2(f)(n+6)^3+2(n+4)||\nabla f(x)||^2.$ These upper bounds can be used as candidates for $\sigma_k^2$.
	\end{rem}
	\begin{theorem}
		Consider problem \eqref{eq:eq87}. Let the sequence $\{x_k\}_{k\geq0}$ be generated by RSc$_\mu$, when $f(x)\in C^{1,1}$ satisfies the proximal PL condition in the sense of Definition \ref{as}. Then,  for any $N\geq0$, with $h_k = \frac{1}{L_{1}(f)}$, we have
		\begin{align}\label{eq:eq126}
			\frac{1}{N+1}\sum_{k=0}^{N}&\Phi_k-f(x^\ast)\leq \frac{L_{1}(f)}{l}\frac{f(x_0)-f(x^\ast)}{N+1} \notag \\
			&+\frac{\mu d_xL_{1}(f)^2(n+3)^{3/2}}{2l}+\frac{L_{1}(f)d_x}{l(N+1)}\sum_{k=0}^{N}\sigma_k \notag\\
			&+\frac{1}{l(N+1)}\sum_{k=0}^{N}\sigma_k^2,
		\end{align}
		where $\Phi_k \overset{\mathrm{def}}{=} E_{\mathcal{U}_{k}}[f(x_k)]$ for all $k\geq 1$, $\Phi_0 = f(x_0)$, $\sigma_k$ is given in Lemma \ref{lm20}, $\mathcal{U}_{k}=\{u_{0},u_{1},\dots,u_{k-1}\}$, $n$ is the dimension of $x$, and $d_x$ is the diameter of the feasible set.
	\end{theorem}
	\begin{proof} From \cite[Equation (12)]{nesterov_random_2017} we know $f_\mu(x) \in C^{1,1}$. Writing \eqref{eq:eq2} for $f_{\mu}(x)$ at points $x_k$ and $x_{k+1}$, yields
		\begin{align*}
			\begin{split}
				&f_{\mu}(x_{k+1})\leq f_{\mu}(x_{k}) + \frac{L_{1}(f_{\mu})}{2}||x_{k+1}-x_k||^2\\&+\langle\nabla f_{\mu}(x_k),x_{k+1}-x_k\rangle\\&\leq f_{\mu}(x_{k}) + \frac{L_{1}(f_{\mu})}{2}||x_{k+1}-x_k||^2\\&+\langle g_{\mu}(x_k),x_{k+1}-x_k\rangle-\langle \xi_{k},x_{k+1}-x_k\rangle\\&\leq f_{\mu}(x_{k})-\frac{1}{2L_{1}(f)}T(x_k,L_{1}(f))+h_k\langle \xi_{k},s_k\rangle\\&\leq f_{\mu}(x_{k})-\frac{1}{2L_{1}(f)}T(x_k,L_{1}(f))\\&+h_k\langle \xi_{k},s_k-v_k\rangle+h_k\langle \xi_{k},v_k\rangle,
			\end{split}
		\end{align*}
		where the third inequality follows from \eqref{eq:T_def}, $x_k\in\mathcal{X}$, and the fact that 
		\begin{align*}
			T(x_k,L_{1}(f))  & = -2 L_{1}(f)  \underset{z\in\mathcal{X}}{\min}\big \{\frac{L_{1}(f)}{2}||z-x_k||^2 \\ & \qquad +\langle g_{\mu}(x_k),z-x_k\rangle \big \}\\ & = -2 L_{1}(f) \big (  \frac{L_{1}(f)}{2}||x_{k+1}-x_k||^2\\ &  \qquad +\langle g_{\mu}(x_k),x_{k+1}-x_k\rangle \big ).
		\end{align*}
		The last equality above follows from the definition of the projection operator.
		Taking expected value with respect to $u_{k}$ and considering Lemmas \ref{lm20} and \ref{lm:gh} in the appendix, and $h_k=\frac{1}{L_{1}(f)}$, leads to
		\begin{align*}
			\begin{split}
				E_{u_{k}}[f_{\mu}&(x_{k+1})]\leq f_{\mu}(x_{k})\\& - \frac{1}{2L_{1}(f)}E_{u_{k}}[T(x_k,L_{1}(f))]+\frac{\sigma_k^2}{L_{1}(f)}.
			\end{split}
		\end{align*}
		Taking the expectation with respect to $\mathcal{U}_{k-1}$, we have
		$$E_{\mathcal{U}_{k}}[T(x_k,L_{1}(f))]\leq2L_{1}(f)(\Phi_{k+1}-\Phi_k)+2\sigma_k^2.$$
		Summing over $k=0$ to $k=N$ and dividing it by $N+1$, results in
		\begin{align*}
			\frac{1}{N+1}\sum_{k=0}^{N}E_{\mathcal{U}_{k}}&[T(x_k,L_{1}(f))]\leq \frac{2}{N+1}\sum_{k=0}^{N}\sigma_k^2+\\&2L_{1}(f)\frac{f(x_0)-f(x^\ast)}{N+1}.
		\end{align*}
		Also, taking the expectation of \eqref{eq:eqPPL} with respect to $u_{k}$ and then $\mathcal{U}_{k-1}$, summing over $k=0$ to $k=N$, dividing it by $N+1$ and using Lemma~\ref{lm:sd} from the appendix, yield
		\begin{align*}
			\begin{split}
				&\frac{2l}{N+1}\sum_{k=0}^{N}\Phi_k-f(x^\ast) \leq  \mu d_xL_{1}(f)^2(n+3)^{3/2}\\&+\frac{1}{N+1}\sum_{k=0}^{N}E_{\mathcal{U}_{k}}[T(x_k,L_{1}(f))]+\frac{2L_{1}(f)d_x}{N+1}\sum_{k=0}^{N}\sigma_k.
			\end{split}
		\end{align*}
		Thus, we have
		\begin{align*}
			\begin{split}
				\frac{1}{N+1}\sum_{k=0}^{N}&\Phi_k-f(x^\ast)\leq \frac{L_{1}(f)}{l}\frac{f(x_0)-f(x^\ast)}{N+1}\\&+\frac{\mu d_xL_{1}(f)^2(n+3)^{3/2}}{2l}+\frac{L_{1}(f)d_x}{l(N+1)}\sum_{k=0}^{N}\sigma_k\\&+\frac{1}{l(N+1)}\sum_{k=0}^{N}\sigma_k^2.
			\end{split}
		\end{align*}
	\end{proof}
	In a practical implementation of the algorithm we might be interested in keeping track of the best guess for the optimum solution at any given step $N$. As in the unconstrained case, let this best guess be denoted by $\hat{x}_N = \arg\underset{x}{\min}[f(x):x\in\{x_0,\dots,x_N\}]$. Thus,
	\begin{equation*}
		E_{\mathcal{U}_{N-1}}(f(\hat{x}_N)-f^\ast )\leq 	\frac{1}{N+1}\sum_{k=0}^{N}(\Phi_k-f^\ast).
	\end{equation*}
	\begin{rem}[$\mathrm{RSc}_{\mu}$ Complexity, Parameter Selection, and Solution Error Bound]\label{remLc}
		If $\mu\leq\frac{l\epsilon}{d_xL_{1}(f)^2(n+3)^{3/2}}$, $N$ is in the order of $\mathcal{O}\left (\frac{L_{1}(f)}{l\epsilon} \right )$, and denoting $\sigma = \underset{k}{\max}[\sigma_k:k\in\{0,\dots,N\}]$, then
		\begin{align}	\label{eq:upperbound_constrained}
			E_{\mathcal{U}_{N-1}}(f(\hat{x}_N)-f^\ast )\leq\epsilon + \frac{L_{1}(f)d_x\sigma}{l}+\frac{\sigma^2}{l},
		\end{align}
		for some positive scalar $\epsilon$. 
		Thus, we can guarantee that there exists an integer $\bar{N}$ such that for all $N\geq \bar{N}$, $E_{\mathcal{U}_{N-1}} f(x_k)$ is in a neighbourhood of $f^\ast$ for an appropriate choice of $\mu$.
	\end{rem}
	Similar phenomena are observed consistently in the literature on constrained non-convex problems, where this effect has been reported for different algorithms, see e.g. \cite{ghadimi2016mini}, \cite{liu2018zeroth}. Note that similarly to the unconstrained case $l$ appears in the denominator of the iteration number order term. Additionally, the two terms on the right-hand side of \eqref{eq:upperbound_constrained} are inversely proportional with $l$. 
	Also, it can be seen in this case due to the choice of $h_k$, the number of iterations is not dependent on the dimension of $x$, but we know that $h_k \in (0,\frac{1}{L_1(f)}]$ and in fact $N$ is in the order of $\mathcal{O}(\frac{1}{lh_k\epsilon})$ for the more general case.

	\section{Numerical Examples}\label{sec:nuemrical}
	In this section, we consider two scenarios. In both scenarios we study the performance of Algorithms \ref{alg:RM} and \ref{alg:RMc} applied to unconstrained and constrained least square problems. Specifically, the objective function is assumed to be $||Ax-b||^2$, where $A\in R^{m \times n}$ ($n\geq m$), $x\in R^n$, and $b\in R^m$. The function is not strongly convex but it satisfies the PL condition with $l = 2||A^T||^2$ and is in $C^{1,1}$ with $L_1(f)=2||A^TA||$.
	
	\paragraph*{Scenario 1} In the first scenario, we consider an unconstrained least-squares problem of the form introduced above with $m = 100$ and $n = 1000$.  In light of Remark~\ref{remCU}, we choose $\epsilon=0.01$ and consequently we set $\mu=10^{-7}$ and $N=200000$. Matrix $A$ rows are sampled from $\mathcal{N}(0,I_m)$  and $b = A\bar{x} + \omega$, where $\bar{x}$ sampled from $\mathcal{N}(0, 1)$ and $\omega$ from $\mathcal{N}(0, 0.01)$. Moreover, the initial condition vector is sampled from $\mathcal{N}(0,1)$.  We explore the performance of the algorithm for two step sizes $\frac{1}{4(n+4)L_1(f)}\approx10^{-7}$ and $10^{-6}$. We repeat the example for 25 times. 
	The empirical mean of best guess for the optimum solution ($f(\hat{x}_N)$) over 25 runs and upper bound calculated in \eqref{eq:eq17} is presented in Fig.~\ref{fig:upper_bound}. As it was shown in Theorem~\ref{thm:unconstrained}, by increasing the step size the convergence to a neighbourhood of the solution would be faster, but this comes at the expense of larger error bounds. 
	
	\paragraph*{Scenario 2}  In this scenario we consider a constrained least squares problem with the same problem parameter choices as the previous scenario except for $\mu$. From Remark~\ref{remLc}, given the same value for $\epsilon$ as the previous case, we choose $\mu=10^{-10}$ and $N=200000$. The constraint set $\mathcal{X}$ is assumed to be $\mathcal{X}=\{x\in\mathbb{R}^n| x_i\in[-0.5,0.5], \forall i\in\{1,\dots,n\} \}$. We explore the performance of the algorithm for two step sizes $\frac{1}{nL_1(f)}\approx10^{-7}$ and $10^{-6}$ (both are less than $\frac{1}{L_1(f)}$). 
	The empirical mean of objective function value in iterations generated by Algorithm~\ref{alg:RMc} over $25$ runs are depicted in Fig.~\ref{fig:obj_func2}. The effect of additional error terms in the constrained case can be seen by comparing the figures at point $10^4$ iterations for $h_k = 10^{-6}$ cases and at point $10^5$ iterations for $h_k = 10^{-7}$cases. 
	
	\begin{figure}
		\centering
		\includegraphics[width=0.33\textwidth]{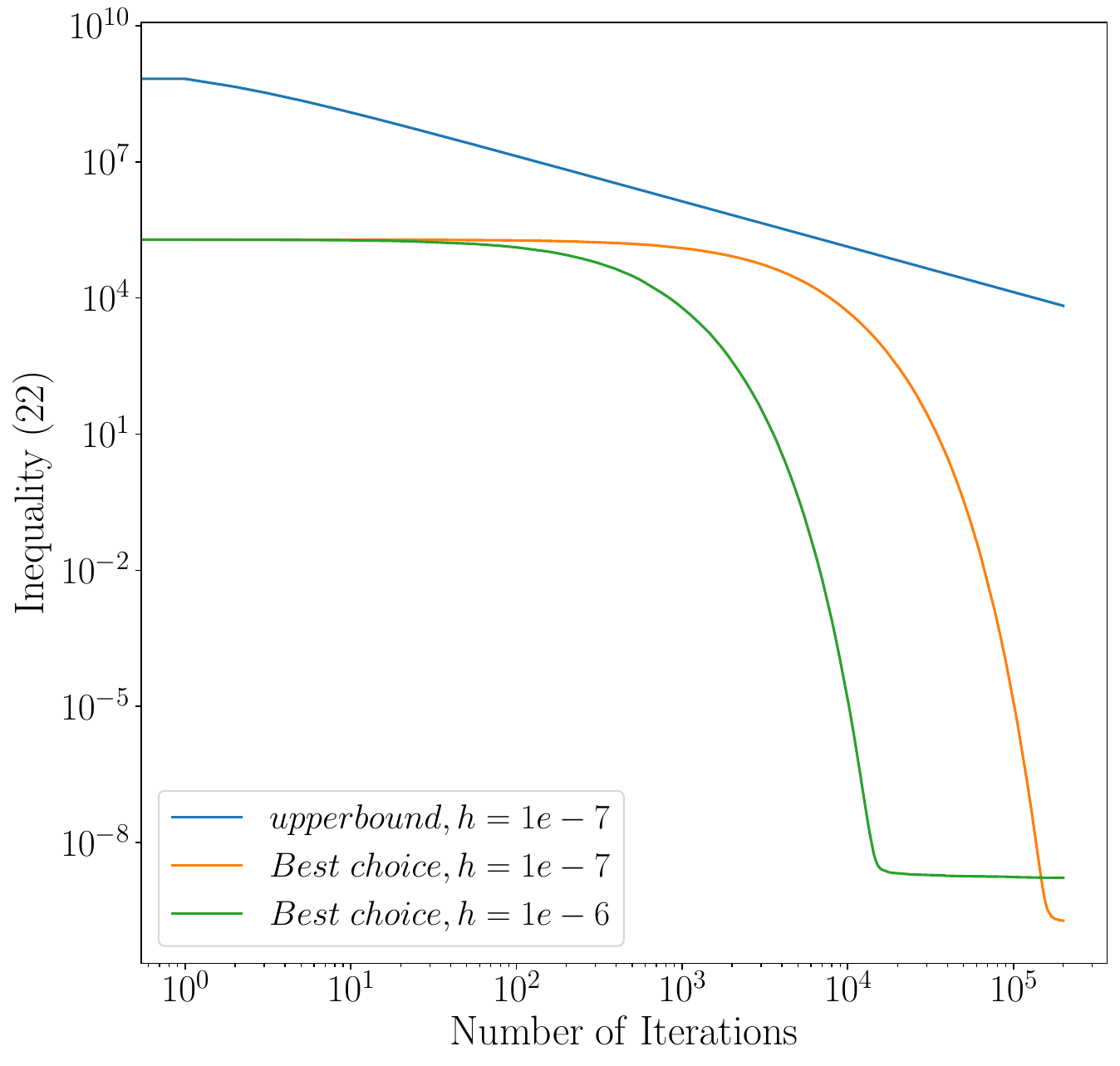}
		\caption{The evolution of the empirical mean of $f(\hat{x}_k)$ and the calculated upper bound versus the number of iterations in Scenario 1.}\label{fig:upper_bound}
	\end{figure}
	\begin{figure}
		\centering
		\includegraphics[width=0.33\textwidth]{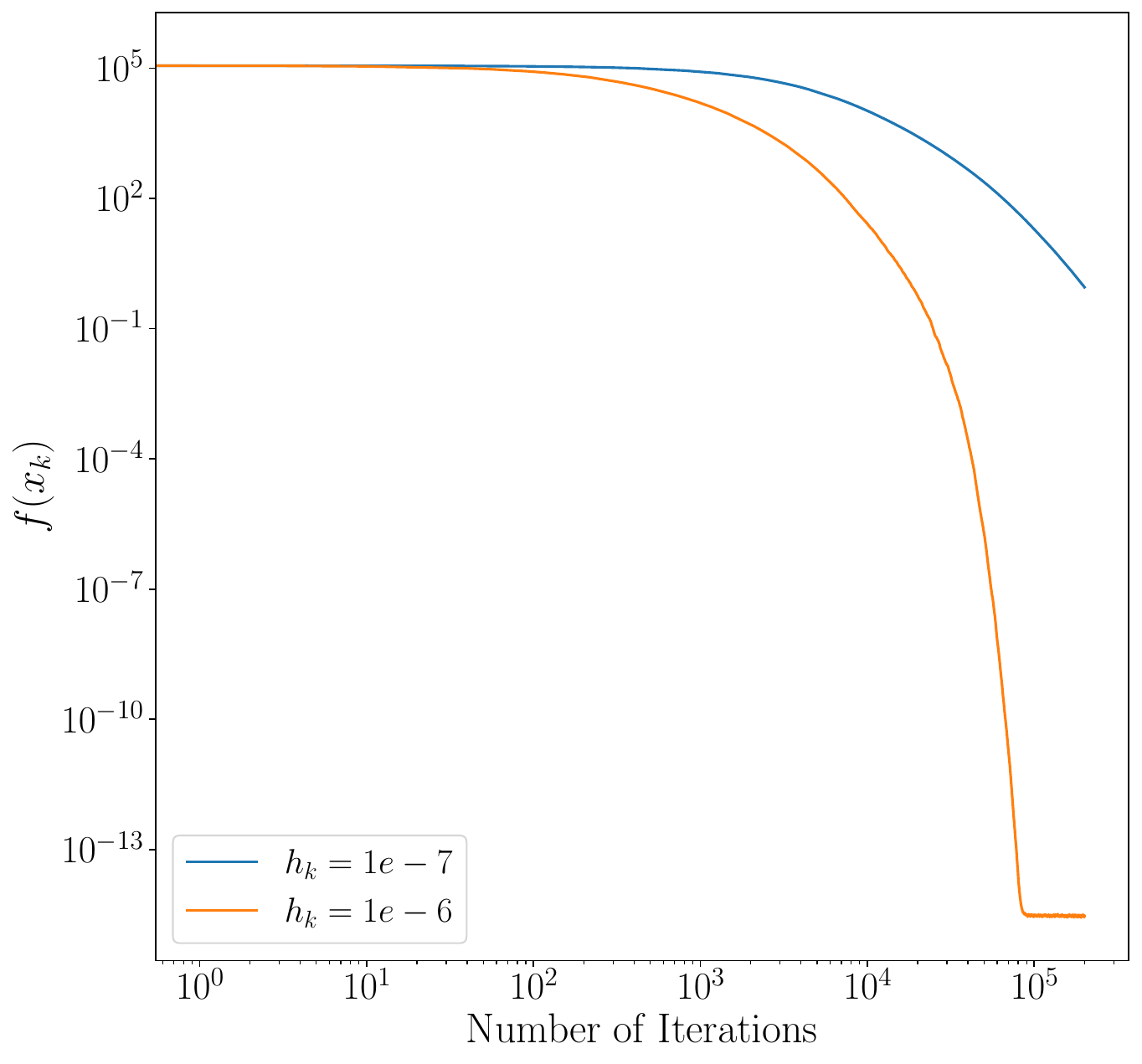}
		\caption{The evolution of $f(x_k)$ versus the number of iterations. Note that in this case $f(x^\ast)=0$ in Scenario 2.  }\label{fig:obj_func2}
	\end{figure}

	\section{Conclusion and Future Research Directions}\label{sec:conc}
	The application of a zeroth-order scheme using  random oracles for minimising PL functions with or without constraints was discussed. For the unconstrained problem, the convergence properties of the proposed algorithm were studied. Additionally, the complexity bounds for the algorithm along with guidelines for selecting algorithm parameters were introduced. Next, a generalisation of the PL inequality was exploited to establish the convergence of the algorithm for solving constrained problems. Similar to the unconstrained case, complexity bounds were derived. Numerical examples were presented to illustrate the theoretical results. An immediate future step is extending the analysis of the constrained case to the case where the constraint set is unbounded. Another possible future direction is applying similar techniques for solving minimax problems using zeroth-order oracles of the type studied in this paper.
	
	\appendix
	
	\subsection{Auxiliary Lemmas}

	\begin{lemma}\label{lm:gh}
		Let $\xi_{k} \overset{\mathrm{def}}{=} g_{\mu}(x_k) -\nabla f_{\mu}(x_k).$
		From \cite[Proposition 1]{ghadimi2016mini}, it implies that $\langle\xi_{k},s_k-v_k\rangle\leq||\xi_{k}||^2$.
	\end{lemma}

	\begin{lemma}\label{lm:sd}
		For a function $f(x)$ with smoothed version $f_{\mu}(x)$ and random oracle $g_{\mu}(x)$, we have 
		$$E_{u_k}[||\xi_{k}||]\leq\sigma_k,\;\;\sigma_k\geq0,$$
		where $\xi_{k} = g_{\mu}(x_k) - \nabla f_{\mu}(x_k)$.
	\end{lemma}
	\begin{proof}
		Due to Lemma \ref{lm20}, $E_{u_k}[||\xi_{k}||^2]\leq\sigma_k^2.$
		From Jensen's inequality, we have $$E_{u_k}[||\xi_{k}||]^2 \leq E_{u_k}[||\xi_{k}||^2].$$
		Thus, 
		$E_{u_k}[||\xi_{k}||]\leq\sigma_k.$
	\end{proof}
	
	\printbibliography
\end{document}